\documentclass[reqno]{amsart}
\usepackage{amsmath,amsthm,amssymb,amscd}

\newtheorem{theorem}{Theorem}[section]
\newtheorem{lemma}[theorem]{Lemma}

\theoremstyle{definition}
\newtheorem{definition}[theorem]{Definition}

\theoremstyle{remark}
\newtheorem{remark}[theorem]{Remark}
\newtheorem{conj}[theorem]{Conjecture}
\numberwithin{equation}{section}
\allowdisplaybreaks
\begin{document}

%
%
%
%
%
%
%
%
%

\title[Divisibility of certain $\ell$-regular partitions by $2$]
 {Divisibility of certain $\ell$-regular partitions by $2$}

\author{Ajit Singh}
\address{Department of Mathematics, Indian Institute of Technology Guwahati, Assam, India, PIN- 781039}
\email{ajit18@iitg.ac.in}

\author{Rupam Barman}
\address{Department of Mathematics, Indian Institute of Technology Guwahati, Assam, India, PIN- 781039}
\email{rupam@iitg.ac.in}

\date{October 27, 2021}




\dedicatory{}

\begin{abstract} For a positive integer $\ell$, let $b_{\ell}(n)$ denote the number of $\ell$-regular partitions of a nonnegative integer $n$. Motivated by some recent conjectures of Keith and Zanello, we establish infinite families of congruences modulo $2$ for $b_3(n)$ and $b_{21}(n)$. We prove a specific case of a conjecture of Keith and Zanello on self-similarities of $b_3(n)$ modulo $2$. We next prove that the  series $\sum_{n=0}^{\infty}b_9(2n+1)q^n$ is lacunary modulo arbitrary powers of $2$. We also prove that 
the  series $\sum_{n=0}^{\infty}b_9(4n)q^n$ is lacunary modulo $2$. 
\end{abstract}

\maketitle
\section{Introduction and statement of results} 
 A partition of a positive integer $n$ is any non-increasing sequence of positive integers whose sum is $n$. The number of such partitions of $n$ is denoted by $p(n)$. The partition function has many congruence properties modulo primes and powers of primes. Ramanujan discovered beautiful congruences satisfied by $p(n)$ modulo $5, 7$ and $11$. Let $\ell$ be a fixed positive integer. An $\ell$-regular partition of a positive integer $n$ is a partition of $n$ such that none of its part is divisible by $\ell$. Let $b_{\ell}(n)$ be the number of $\ell$-regular partitions of $n$. The generating function for $b_{\ell}(n)$ is given by 
\begin{align}\label{gen_fun}
G_{\ell}(q):=\sum_{n=0}^{\infty}b_{\ell}(n)q^n=\frac{f_{\ell}}{f_1},
\end{align}
where $f_k:=(q^k; q^k)_{\infty}=\prod_{j=1}^{\infty}(1-q^{jk})$ and $k$ is a positive integer.
 \par 
 In a very recent paper \cite{Keith2021}, Keith and Zanello studied $\ell$-regular partition for certain values of $\ell$. They proved various congruences for $b_3(n)$ and make the following conjecture regarding the self-similarity of $b_3(n)$.
\begin{conj}\label{conj1}
For any prime $p>3$, let $\gamma\equiv-24^{-1}\pmod{p^2}$, $0<\gamma<p^2$. It holds for a positive proportion of primes $p$ that 
\begin{align}\label{new-eq-2}
\sum_{n=0}^{\infty}b_{3}(2(pn+\gamma))q^n\equiv \sum_{n=0}^{\infty}b_{3}(2n)q^{pn}\pmod 2.
\end{align}
\end{conj}
In \cite[Theorem 7]{Keith2021}, Keith and Zanello proved one specific case of Conjecture \ref{conj1} corresponding to $p=13$. In the following theorem, we prove another specific case of Conjecture \ref{conj1} corresponding to $p=17$. 
\begin{theorem}\label{thmnew}
	It holds that 
	\begin{align*}
	\sum_{n=0}^{\infty}b_{3}(34n+24)q^n\equiv \sum_{n=0}^{\infty}b_{3}(2n)q^{17n}\pmod 2,
	\end{align*}
	and therefore 
	$$b_3(2\cdot 17^2n+58)\equiv 0\pmod 2,$$
	and by iteration,
	\begin{align*}b_3\left(2\cdot 17^{2k}n+17^{2k-2}\cdot 58+24\left(\frac{17^{2k-2}-1}{288}\right)\right)\equiv 0\pmod 2
	\end{align*}
	for all $k\geq 1$.
\end{theorem}
If we assume that Conjecture \ref{conj1} is true for $p=29$,  then using the congruence $\sum_{n=0}^{\infty}b_{3}(2(29n+35))q^n\equiv \sum_{n=0}^{\infty}b_{3}(2n)q^{29n}\pmod 2$, one can deduce infinite families of congruences of the form 
\begin{align}\label{new-eq-1}
b_{3}(2(29^2\cdot n+29k+35))\equiv 0\pmod 2, 
\end{align}
where $1\leq k\leq 28$.
We do not know whether Conjecture \ref{conj1} is true or not for $p=29$. However, in the following theorem, we prove the congruences \eqref{new-eq-1} without assuming \eqref{new-eq-2} for $p=29$. 
\begin{theorem} \label{thm1}
Let $\alpha\in\{6,64,93,122,151,180,209,238,267,296,325,354,383,412,\\ 441,470,499,528,557,586,615,644,673,
702,731,760,789,818\}$. Then for all $n\geq 0$, we have 
$$b_3(2(29^2\cdot n+\alpha))\equiv 0\pmod 2.$$
\end{theorem}
Keith and Zanello \cite{Keith2021} also studied $2$-divisibility of $b_{21}(n)$ and proved several congruences for primes $p\equiv 13,17,19,23 \pmod{24}$. To be specific, if $p\equiv 13,17,19,23 \pmod{24}$ is prime, then 
\begin{align*}
b_{21}(4(p^2n+kp-11\cdot 24^{-1})+1)\equiv 0\pmod{2}
\end{align*}
for all $1\leq k<p$, where $24^{-1}$ is taken modulo $p^2$. For example, if $p=13$, then one obtains 
\begin{align*}
b_{21}(4\cdot 13^2n+52k+309)\equiv 0 \pmod{2}
\end{align*}
for all $k=1, 2, \ldots, 12$.
In the following theorem we prove similar type of congruences for the prime $p=29$.
\begin{theorem} \label{thm2}
Let $\beta\in\{8,37,66,95,124,153,182,211,240,269,298,327,356,414,\\ 443,472,501,530,559,588,617,646,675,
704,733,762,791,820\}$. Then for all $n\geq 0$, we have 
$$b_{21}(4(29^2\cdot n+\beta)+1)\equiv 0\pmod 2.$$
\end{theorem}
In addition to the study of Ramanujan-type congruences, it is an interesting problem to study the distribution of the partition function modulo positive integers $M$. To be precise, given an integral power series $F(q):=\sum_{n=0}^{\infty}a(n)q^n$ and $0\leq r<M$, we define
\begin{align*}
\delta_r(F, M; X):=\frac{\#\{n\leq X: a(n)\equiv r \pmod{M}\}}{X}.
\end{align*} 
An integral power series $F$ is called \textit{lacunary modulo $M$} if 
\begin{align*}
\lim _{X\rightarrow \infty}\delta_0(F, M; X)=1,
\end{align*}
that is, almost all of the coefficients of $F$ are divisible by $M$. In a recent paper \cite{Cotron2020}, Cotron et al. proved lacunarity of certain eta-quotients modulo arbitrary powers of primes. We phrase their theorem as follows:
\begin{theorem}\cite[Theorem 1.1]{Cotron2020}\label{Cotron}
Let $G(z)=\frac{\prod_{i=1}^{u}f_{\alpha_i}^{r_i}}{\prod_{i=1}^{t}f_{\beta_i}^{s_i}}$, and $p$ is a prime such that $p^a$ divides $\gcd(\alpha_1,
\alpha_2, \ldots, \alpha_u)$ and 
\begin{align*}
p^a\geq \sqrt{\frac{\sum_{i=1}^{t}\beta_i s_i}{\sum_{i=1}^{u}\frac{r_i}{\alpha_i}}},
\end{align*} 
then $G(z)$ is lacunary modulo $p^j$ for any positive integer $j$.
\end{theorem}
Keith and Zanello \cite{Keith2021} studied lacunarity of the functions $b_3(2n)$, $b_{21}(4n)$, $b_{21}(4n+1)$ and $b_{21}(8n+3)$ modulo $2$ using the technique developed by Landau \cite{Landau1908}. In the following theorem, we prove that $b_9(2n+1)$ is almost always divisible by arbitrary powers of $2$.
\begin{theorem} \label{thm3}
The series $\sum_{n=0}^{\infty}b_9(2n+1)q^n$ is lacunary modulo $2^k$ for any positive integer $k$.
\end{theorem}
Keith and Zanello \cite{Keith2021} derived several congruences for the partition function $b_{9}(n)$ modulo $2$ using the theory of Hecke operators. In the following theorem we prove that $b_9(4n)$ is almost always divisible by $2$.
\begin{theorem} \label{thm4}
	The series $\sum_{n=0}^{\infty}b_9(4n)q^n$ is lacunary modulo $2$.
\end{theorem}
We prove Theorem \ref{thm4} using the approach of Landau \cite{Landau1908} and Theorem \ref{Cotron}. However, we couldn't find a similar proof for Theorem \ref{thm3}. We use a density result of Serre to prove Theorem \ref{thm3}. 
\section{Preliminaries}
We recall some definitions and basic facts on modular forms. For more details, see for example \cite{koblitz1993, ono2004}. We first define the matrix groups 
\begin{align*}
\text{SL}_2(\mathbb{Z}) & :=\left\{\begin{bmatrix}
a  &  b \\
c  &  d      
\end{bmatrix}: a, b, c, d \in \mathbb{Z}, ad-bc=1
\right\},\\
\Gamma_{\infty} & :=\left\{
\begin{bmatrix}
1  &  n \\
0  &  1      
\end{bmatrix}: n\in \mathbb{Z}  \right\},\\
\Gamma_{0}(N) & :=\left\{
\begin{bmatrix}
a  &  b \\
c  &  d      
\end{bmatrix} \in \text{SL}_2(\mathbb{Z}) : c\equiv 0\pmod N \right\},
\end{align*}
\begin{align*}
\Gamma_{1}(N) & :=\left\{
\begin{bmatrix}
a  &  b \\
c  &  d      
\end{bmatrix} \in \Gamma_0(N) : a\equiv d\equiv 1\pmod N \right\},
\end{align*}
and 
\begin{align*}\Gamma(N) & :=\left\{
\begin{bmatrix}
a  &  b \\
c  &  d      
\end{bmatrix} \in \text{SL}_2(\mathbb{Z}) : a\equiv d\equiv 1\pmod N, ~\text{and}~ b\equiv c\equiv 0\pmod N\right\},
\end{align*}
where $N$ is a positive integer. A subgroup $\Gamma$ of $\text{SL}_2(\mathbb{Z})$ is called a congruence subgroup if $\Gamma(N)\subseteq \Gamma$ for some $N$. The smallest $N$ such that $\Gamma(N)\subseteq \Gamma$
is called the level of $\Gamma$. For example, $\Gamma_0(N)$ and $\Gamma_1(N)$
are congruence subgroups of level $N$. 
\par Let $\mathbb{H}:=\{z\in \mathbb{C}: \text{Im}(z)>0\}$ be the upper half of the complex plane. The group $$\text{GL}_2^{+}(\mathbb{R})=\left\{\begin{bmatrix}
a  &  b \\
c  &  d      
\end{bmatrix}: a, b, c, d\in \mathbb{R}~\text{and}~ad-bc>0\right\}$$ acts on $\mathbb{H}$ by $\begin{bmatrix}
a  &  b \\
c  &  d      
\end{bmatrix} z=\displaystyle \frac{az+b}{cz+d}$.  
We identify $\infty$ with $\displaystyle\frac{1}{0}$ and define $\begin{bmatrix}
a  &  b \\
c  &  d      
\end{bmatrix} \displaystyle\frac{r}{s}=\displaystyle \frac{ar+bs}{cr+ds}$, where $\displaystyle\frac{r}{s}\in \mathbb{Q}\cup\{\infty\}$.
This gives an action of $\text{GL}_2^{+}(\mathbb{R})$ on the extended upper half-plane $\mathbb{H}^{\ast}=\mathbb{H}\cup\mathbb{Q}\cup\{\infty\}$. 
Suppose that $\Gamma$ is a congruence subgroup of $\text{SL}_2(\mathbb{Z})$. A cusp of $\Gamma$ is an equivalence class in $\mathbb{P}^1=\mathbb{Q}\cup\{\infty\}$ under the action of $\Gamma$.
\par The group $\text{GL}_2^{+}(\mathbb{R})$ also acts on functions $f: \mathbb{H}\rightarrow \mathbb{C}$. In particular, suppose that $\gamma=\begin{bmatrix}
a  &  b \\
c  &  d      
\end{bmatrix}\in \text{GL}_2^{+}(\mathbb{R})$. If $f(z)$ is a meromorphic function on $\mathbb{H}$ and $\ell$ is an integer, then define the slash operator $|_{\ell}$ by 
$$(f|_{\ell}\gamma)(z):=(\text{det}~{\gamma})^{\ell/2}(cz+d)^{-\ell}f(\gamma z).$$
\begin{definition}
	Let $\Gamma$ be a congruence subgroup of level $N$. A holomorphic function $f: \mathbb{H}\rightarrow \mathbb{C}$ is called a modular form with integer weight $\ell$ on $\Gamma$ if the following hold:
	\begin{enumerate}
		\item We have $$f\left(\displaystyle \frac{az+b}{cz+d}\right)=(cz+d)^{\ell}f(z)$$ for all $z\in \mathbb{H}$ and all $\begin{bmatrix}
		a  &  b \\
		c  &  d      
		\end{bmatrix} \in \Gamma$.
		\item If $\gamma\in \text{SL}_2(\mathbb{Z})$, then $(f|_{\ell}\gamma)(z)$ has a Fourier expansion of the form $$(f|_{\ell}\gamma)(z)=\displaystyle\sum_{n\geq 0}a_{\gamma}(n)q_N^n,$$
		where $q_N:=e^{2\pi iz/N}$.
	\end{enumerate}
\end{definition}
For a positive integer $\ell$, the complex vector space of modular forms of weight $\ell$ with respect to a congruence subgroup $\Gamma$ is denoted by $M_{\ell}(\Gamma)$.
\begin{definition}\cite[Definition 1.15]{ono2004}
	If $\chi$ is a Dirichlet character modulo $N$, then we say that a modular form $f\in M_{\ell}(\Gamma_1(N))$ has Nebentypus character $\chi$ if
	$$f\left( \frac{az+b}{cz+d}\right)=\chi(d)(cz+d)^{\ell}f(z)$$ for all $z\in \mathbb{H}$ and all $\begin{bmatrix}
	a  &  b \\
	c  &  d      
	\end{bmatrix} \in \Gamma_0(N)$. The space of such modular forms is denoted by $M_{\ell}(\Gamma_0(N), \chi)$. 
\end{definition}
In this paper, the relevant modular forms are those that arise from eta-quotients. Recall that the Dedekind eta-function $\eta(z)$ is defined by
\begin{align*}
\eta(z):=q^{1/24}(q;q)_{\infty}=q^{1/24}\prod_{n=1}^{\infty}(1-q^n),
\end{align*}
where $q:=e^{2\pi iz}$ and $z\in \mathbb{H}$. A function $f(z)$ is called an eta-quotient if it is of the form
\begin{align*}
f(z)=\prod_{\delta\mid N}\eta(\delta z)^{r_\delta},
\end{align*}
where $N$ is a positive integer and $r_{\delta}$ is an integer. We now recall two theorems from \cite[p. 18]{ono2004} which will be used to prove our results.
\begin{theorem}\cite[Theorem 1.64]{ono2004}\label{thm_ono1} If $f(z)=\prod_{\delta\mid N}\eta(\delta z)^{r_\delta}$ 
	is an eta-quotient such that $\ell=\frac{1}{2}\sum_{\delta\mid N}r_{\delta}\in \mathbb{Z}$, 
	$$\sum_{\delta\mid N} \delta r_{\delta}\equiv 0 \pmod{24}$$ and
	$$\sum_{\delta\mid N} \frac{N}{\delta}r_{\delta}\equiv 0 \pmod{24},$$
	then $f(z)$ satisfies $$f\left( \frac{az+b}{cz+d}\right)=\chi(d)(cz+d)^{\ell}f(z)$$
	for every  $\begin{bmatrix}
	a  &  b \\
	c  &  d      
	\end{bmatrix} \in \Gamma_0(N)$. Here the character $\chi$ is defined by $\chi(d):=\left(\frac{(-1)^{\ell} s}{d}\right)$, where $s:= \prod_{\delta\mid N}\delta^{r_{\delta}}$. 
\end{theorem}
Suppose that $f$ is an eta-quotient satisfying the conditions of Theorem \ref{thm_ono1} and that the associated weight $\ell$ is a positive integer. If $f(z)$ is holomorphic at all of the cusps of $\Gamma_0(N)$, then $f(z)\in M_{\ell}(\Gamma_0(N), \chi)$. The following theorem gives the necessary criterion for determining orders of an eta-quotient at cusps.
\begin{theorem}\cite[Theorem 1.65]{ono2004}\label{thm_ono2}
	Let $c, d$ and $N$ be positive integers with $d\mid N$ and $\gcd(c, d)=1$. If $f$ is an eta-quotient satisfying the conditions of Theorem \ref{thm_ono1} for $N$, then the 
	order of vanishing of $f(z)$ at the cusp $\frac{c}{d}$ 
	is $$\frac{N}{24}\sum_{\delta\mid N}\frac{\gcd(d,\delta)^2r_{\delta}}{\gcd(d,\frac{N}{d})d\delta}.$$
\end{theorem}
We now recall a result of Sturm \cite{Sturm1984} which gives a criterion to test whether two modular forms are congruent modulo a given prime.
\begin{theorem}\label{Sturm}
	Let $p$ be a prime number, and $f(z)=\sum_{n=n_0}^\infty a(n)q^n$ and $g(z)=\sum_{n=n_1}^\infty b(n)q^n$ be modular forms of weight $k$ for $\Gamma_0(N)$ of characters $\chi$ and $\psi$, respectively, where $n_0, n_1\geq 0$. If either $\chi=\psi$ and 
	\begin{align*}
	a(n)\equiv b(n)\pmod p~~ \text{for all}~~ n\leq \frac{kN}{12}\prod_{d~~ \text{prime};~~ d|N}\left(1+\frac{1}{d}\right),
	\end{align*}
	or $\chi\neq\psi$ and 
	\begin{align*}
	a(n)\equiv b(n)\pmod p~~ \text{for all}~~ n\leq \frac{kN^2}{12}\prod_{d~~ \text{prime};~~ d|N}\left(1-\frac{1}{d^2}\right),
	\end{align*}
	then $f(z)\equiv g(z)\pmod p$~~ $(i.e.,~~a(n)\equiv b(n)\pmod p~~\text{for all}~~n)$.
\end{theorem}
\par We next recall the definition of Hecke operators.
Let $m$ be a positive integer and $f(z) = \sum_{n=0}^{\infty} a(n)q^n \in M_{\ell}(\Gamma_0(N),\chi)$. Then the action of Hecke operator $T_m$ on $f(z)$ is defined by 
\begin{align*}
f(z)|T_m := \sum_{n=0}^{\infty} \left(\sum_{d\mid \gcd(n,m)}\chi(d)d^{\ell-1}a\left(\frac{nm}{d^2}\right)\right)q^n.
\end{align*}
In particular, if $m=p$ is prime, we have 
\begin{align*}
f(z)|T_p := \sum_{n=0}^{\infty} \left(a(pn)+\chi(p)p^{\ell-1}a\left(\frac{n}{p}\right)\right)q^n.
\end{align*}
We take by convention that $a(n/p)=0$ whenever $p \nmid n$.
If $f$ is an $\eta$-quotient with the properties listed in Theorem \ref{thm_ono1}, and $p|s$ (here $s$ is as defined in Theorem \ref{thm_ono1}), then $\chi(p)=0$ so that the latter term vanishes. In this case, we have the factorization property that
\begin{align*}
\left(f\cdot\sum_{n=0}^\infty g(n)q^{pn}\right)|T_p=\left(\sum_{n=0}^\infty a(pn)q^{n}\right)\left(\sum_{n=0}^\infty g(n)q^{n}\right).
\end{align*}
\par We finally recall a density result of Serre \cite{serre3}  about the divisibility of Fourier coefficients of modular forms.
\begin{theorem}[Serre]\label{Serre}
	Let $f(z)$ be a modular form of positive integer weight $k$ on some congruence subgroup of $SL_2(\mathbb{Z})$ with Fourier expansion $$f(z)=\sum_{n=0}^{\infty}a(n)q^n,$$
	where $a(n)$ are algebraic integers in some number field. If $m$ is a positive integer, then there exists a constant $c>0$ such that there are  $O\left(\frac{X}{(\log X)^{c}}\right)$ integers $n \leq  X$ such that
	$a(n)$ is not divisible by $m$.
\end{theorem}
\section{Proof of Theorems \ref{thmnew}}
\begin{proof}
	We first recall the following even-odd disection of the $3$-regular partitions \cite[(6)]{Keith2021}:
	\begin{align}\label{eqn-new-3}
	\sum_{n=0}^{\infty}b_3(n)q^n=\frac{f_3}{f_1}\equiv\frac{f_1^8}{f_3^2}+q\frac{f_3^{10}}{f_1^4}\pmod 2.
	\end{align}
	Thus, extracting the terms with even powers of $q$, we obtain
	\begin{align}\label{thmnew1}
	\sum_{n=0}^{\infty}b_3(2n)q^n\equiv \frac{f_1^4}{f_3}\pmod 2.
	\end{align} 
	Let $$G_{3,1}(z):= \frac{\eta^4(z)\eta^2(51z)\eta(17z)}{\eta(3z)}$$
	and $$G_{3,2}(z):= \frac{\eta^4(17z)\eta^2(3z)\eta(z)}{\eta(51z)}.$$
	By Theorems \ref{thm_ono1} and \ref{thm_ono2} we find that $G_{3,1}(z)$ and $G_{3,2}(z)$ are modular forms of weight $3$, level $51$ and character $\chi_0=(\frac{-3\cdot 17^{3}}{\bullet})$. 
	By \eqref{thmnew1} the Fourier expansions of our forms satisfy
	\begin{align*}
	G_{3,1}(z)=\left(\sum_{n=0}^{\infty}b_3(2n)q^{n+5}\right)f^2_{51}f_{17}
	\end{align*}
	and 
	\begin{align*}
	G_{3,2}(z)=\left(\sum_{n=0}^{\infty}b_3(2n)q^{17n+1}\right)f^2_{3}f_{1}.
	\end{align*}
	We then calculate that 
	\begin{align*}
	G_{3,1}(z)|T_{17}\equiv \left(\sum_{n=0}^{\infty}b_3(34n+24)q^{n+1}\right)f^2_{3}f_{1}\pmod2.
	\end{align*}
	Since the Hecke operator is an endomorphism on $ M_{3}\left(\Gamma_{0}(51), \chi_0\right)$, we have that $G_{3,1}(z)|T_{17}\in M_{3}\left(\Gamma_{0}(51), \chi_0\right)$. By Theorem \ref{Sturm}, the Sturm bound for this space of forms is $18$. We wish to verify the congruence
	\begin{align*}
	q\left(\sum_{n=0}^{\infty}b_3(34n+24)q^{n}\right)f^2_{3}f_{1}\equiv q\frac{f^4_{17}}{f_{51}} f^2_{3}f_{1}\pmod 2.
	\end{align*} 
	The coefficient of $q^{18}$ on the left side involves the value $b_3(636)$; thus, $f_3/f_1$ must be expanded at least that far, and the product on the right side must be constructed up to the $q^{18}$ terms. Finally, expansion with a calculation package such as $Sage$ confirms that all coefficients up to the desired bound are congruent modulo $2$, and the first part of the theorem is established.\\
	Since only powers for which $17|n$ can be nonzero on the right side of the statement, we
	obtain:
	\begin{align*}
	b_3(34(17n+1)+24)=b_3(2\cdot 17^2n+58)\equiv 0\pmod2.
	\end{align*}
	Finally, recursively applying the relation 
	\begin{align*}
	b_3(2n)\equiv b_3(34\cdot 17n+24)\pmod 2,
	\end{align*}
	we obtain
	\begin{align*}
	b_3(2\cdot 17^2n+58)&\equiv b_3(2\cdot 17^2(17^2n+29)+24)\pmod 2\\
	&=b_3(2\cdot 17^4n+17^2\cdot 58+24)\\
	&\equiv b_3(2\cdot 17^6n+17^4\cdot58+17^2\cdot 24 +24)\pmod 2\\
	&\equiv \ldots\\
	&\equiv b_3\left(2\cdot 17^{2k}n+17^{2k-2}\cdot 58+24\left(\frac{17^{2k-2}-1}{288}\right)\right)\equiv 0\pmod 2,
	\end{align*}
	where the last line is given by a finite geometric summation. This completes the proof of the theorem.
\end{proof}
\section{Proof of Theorems \ref{thm1} and \ref{thm2}}
We prove Theorems \ref{thm1} and \ref{thm2} using the approach developed in \cite{radu1, radu2}. Throughout this section, $\Gamma$ denotes the full modular group $\text{SL}_2(\mathbb{Z})$. We recall that the index of $\Gamma_{0}(N)$ in $\Gamma$ is
\begin{align*}
 [\Gamma : \Gamma_0(N)] = N\prod_{p|N}(1+p^{-1}), 
\end{align*}
where $p$ denotes a prime.
\par 
For a positive integer $M$, let $R(M)$ be the set of integer sequences $r=(r_\delta)_{\delta|M}$ indexed by the positive divisors of $M$. 
If $r \in R(M)$ and $1=\delta_1<\delta_2< \cdots <\delta_k=M$ 
are the positive divisors of $M$, we write $r=(r_{\delta_1}, \ldots, r_{\delta_k})$. Define $c_r(n)$ by 
\begin{align}
\sum_{n=0}^{\infty}c_r(n)q^n:=\prod_{\delta|M}(q^{\delta};q^{\delta})^{r_{\delta}}_{\infty}=\prod_{\delta|M}\prod_{n=1}^{\infty}(1-q^{n \delta})^{r_{\delta}}.
\end{align}
The approach to proving congruences for $c_r(n)$ developed by Radu \cite{radu1, radu2} reduces the number of coefficients that one must check as compared with the classical method which uses Sturm's bound alone.
\par 
Let $m$ be a positive integer. For any integer $s$, let $[s]_m$ denote the residue class of $s$ in $\mathbb{Z}_m:= \mathbb{Z}/ {m\mathbb{Z}}$. 
Let $\mathbb{Z}_m^{*}$ be the set of all invertible elements in $\mathbb{Z}_m$. Let $\mathbb{S}_m\subseteq\mathbb{Z}_m$  be the set of all squares in $\mathbb{Z}_m^{*}$. For $t\in\{0, 1, \ldots, m-1\}$
and $r \in R(M)$, we define a subset $P_{m,r}(t)\subseteq\{0, 1, \ldots, m-1\}$ by
\begin{align*}
P_{m,r}(t):=\left\{t': \exists [s]_{24m}\in \mathbb{S}_{24m} ~ \text{such} ~ \text{that} ~ t'\equiv ts+\frac{s-1}{24}\sum_{\delta|M}\delta r_\delta \pmod{m} \right\}.
\end{align*}
\begin{definition}
	Suppose $m, M$ and $N$ are positive integers, $r=(r_{\delta})\in R(M)$ and $t\in \{0, 1, \ldots, m-1\}$. Let $k=k(m):=\gcd(m^2-1,24)$ and write  
	\begin{align*}
	\prod_{\delta|M}\delta^{|r_{\delta}|}=2^s\cdot j,
	\end{align*}
	where $s$ and $j$  are nonnegative integers with $j$ odd. The set $\Delta^{*}$ consists of all tuples $(m, M, N, (r_{\delta}), t)$ satisfying these conditions and all of the following.
	\begin{enumerate}
		\item Each prime divisor of $m$ is also a divisor of $N$.
		\item $\delta|M$ implies $\delta|mN$ for every $\delta\geq1$ such that $r_{\delta} \neq 0$.
		\item $kN\sum_{\delta|M}r_{\delta} mN/\delta \equiv 0 \pmod{24}$.
		\item $kN\sum_{\delta|M}r_{\delta} \equiv 0 \pmod{8}$.  
		\item  $\frac{24m}{\gcd{(-24kt-k{\sum_{{\delta}|M}}{\delta r_{\delta}}},24m)}$ divides $N$.
		\item If $2|m$, then either $4|kN$ and $8|sN$ or $2|s$ and $8|(1-j)N$.
	\end{enumerate}
\end{definition}
Let $m, M, N$ be positive integers. For $\gamma=
\begin{bmatrix}
	a  &  b \\
	c  &  d     
\end{bmatrix} \in \Gamma$, $r\in R(M)$ and $r'\in R(N)$, set 
	\begin{align*}
	p_{m,r}(\gamma):=\min_{\lambda\in\{0, 1, \ldots, m-1\}}\frac{1}{24}\sum_{\delta|M}r_{\delta}\frac{\gcd^2(\delta a+ \delta k\lambda c, mc)}{\delta m}
	\end{align*}
and 
\begin{align*}
	p_{r'}^{*}(\gamma):=\frac{1}{24}\sum_{\delta|N}r'_{\delta}\frac{\gcd^2(\delta, c)}{\delta}.
\end{align*}
\begin{lemma}\label{lem1}\cite[Lemma 4.5]{radu1} Let $u$ be a positive integer, $(m, M, N, r=(r_{\delta}), t)\in\Delta^{*}$ and $r'=(r'_{\delta})\in R(N)$. 
Let $\{\gamma_1,\gamma_2, \ldots, \gamma_n\}\subseteq \Gamma$ be a complete set of representatives of the double cosets of $\Gamma_{0}(N) \backslash \Gamma/ \Gamma_\infty$. 
Assume that $p_{m,r}(\gamma_i)+p_{r'}^{*}(\gamma_i) \geq 0$ for all $1 \leq i \leq n$. Let $t_{min}=\min_{t' \in P_{m,r}(t)} t'$ and
\begin{align*}
\nu:= \frac{1}{24}\left\{ \left( \sum_{\delta|M}r_{\delta}+\sum_{\delta|N}r'_{\delta}\right)[\Gamma:\Gamma_{0}(N)] -\sum_{\delta|N} \delta r'_{\delta}\right\}-\frac{1}{24m}\sum_{\delta|M}\delta r_{\delta} 
	- \frac{ t_{min}}{m}.
\end{align*}	
If the congruence $c_r(mn+t')\equiv0\pmod u$ holds for all $t' \in P_{m,r}(t)$ and $0\leq n\leq \lfloor\nu\rfloor$, then it holds for all $t'\in P_{m,r}(t)$ and $n\geq0$.
\end{lemma}
	To apply Lemma \ref{lem1}, we utilize the following result, which gives us a complete set of representatives of the double coset in $\Gamma_{0}(N) \backslash \Gamma/ \Gamma_\infty$. 
\begin{lemma}\label{lem2}\cite[Lemma 4.3]{wang} If $N$ or $\frac{1}{2}N$ is a square-free integer, then
		\begin{align*}
		\bigcup_{\delta|N}\Gamma_0(N)\begin{bmatrix}
		1  &  0 \\
		\delta  &  1      
		\end{bmatrix}\Gamma_ {\infty}=\Gamma.
		\end{align*}
\end{lemma}
\begin{proof}[Proof of Theorem \ref{thm1}]
From \eqref{eqn-new-3}, we have 
\begin{align*}
\sum_{n=0}^{\infty}b_3(n)q^n=\frac{f_3}{f_1}\equiv\frac{f_1^8}{f_3^2}+q\frac{f_3^{10}}{f_1^4}\pmod 2.
\end{align*}
Extracting the terms with even powers of $q$, we obtain
\begin{align*}
\sum_{n=0}^{\infty}b_3(2n)q^n\equiv \frac{f_1^4}{f_3}\pmod 2.
\end{align*}
Let $(m,M,N,r,t)=(841,3,87,(4,-1),64)$. It is easy to verify that $(m,M,N,r,t) \in \Delta^{*}$ and $P_{m,r}(t)=\{6,64,151,180,209,238,296,412,499,615,673,702,731,760\}$.
By Lemma \ref{lem2}, we know that $\left\{\begin{bmatrix}
	1  &  0 \\
	\delta  &  1      
	\end{bmatrix}:\delta|87 \right\}$ forms a complete set of double coset representatives of $\Gamma_{0}(N) \backslash \Gamma/ \Gamma_\infty$.
	Let $\gamma_{\delta}=\begin{bmatrix}
	1  &  0 \\
	\delta  &  1      
	\end{bmatrix}$. Let $r'=(0,0,0,0)\in R(87)$ and we use $Sage$ to verify that
	$p_{m,r}(\gamma_{\delta})+p_{r'}^{*}(\gamma_{\delta}) \geq 0$ for each $\delta | N$. We compute that the upper bound in Lemma \ref{lem1} is $\lfloor\nu\rfloor=14$. Using $Sage$ we verify that $b_3(1682n+2t')\equiv0\pmod{2}$ for all $t' \in P_{m,r}(t)$ and for $n\leq 14$. By Lemma \ref{lem1} we conclude that $b_3(1682n+2t')\equiv0\pmod{2}$ for all $t' \in P_{m,r}(t)$ and for all $n\geq 0$. 
	To prove the remaining congruences, we take $(m,M,N,r,t)=(841,3,87,(4,-1),93)$. It is easy to verify that $(m,M,N,r,t) \in \Delta^{*}$ and $P_{m,r}(t)=\{93,122,267,325,354,383,441,470,528,557,586,644,789,818\}$.
	Following similar steps as shown before, we find that $b_3(1682n+2t')\equiv0\pmod{2}$ for all $t' \in P_{m,r}(t)$ and for all $n\geq 0$. This completes the proof of the theorem.
\end{proof}
\begin{proof}[Proof of Theorem \ref{thm2}]
We begin our proof by recalling the following even-odd disection formula of the $21$-regular partitions \cite[(9)]{Keith2021}:
\begin{align*}
\sum_{n=0}^{\infty}b_{21}(n)q^n=\frac{f_{21}}{f_1}\equiv & f_1^8f_3^4+q^3f_1^8f_{21}^4+q^6\frac{f_1^8f_{21}^8}{f_3^4}+q\frac{f_3^{16}}{f_1^4}\\
&+q^4\frac{f_3^{12}f_{21}^4}{f_1^4}+q^7\frac{f_3^{8}f_{21}^8}{f_1^4}\pmod 2.
\end{align*}
Extracting the terms with odd powers of $q$, we obtain
\begin{align*}
\sum_{n=0}^{\infty}b_{21}(2n+1)q^n\equiv qf_1^4f_{21}^2+\frac{f_3^8}{f_1^2}+q^3\frac{f_3^4f_{21}^4}{f_1^2}\pmod 2.
\end{align*}
Finally, extracting the terms with even powers of $q$, we obtain
\begin{align*}
\sum_{n=0}^{\infty}b_{21}(4n+1)q^n\equiv \frac{f_3^4}{f_1}\pmod 2.
\end{align*}
Let $(m,M,N,r,t)=(841,3,87,(-1,4),414)$. We verify that $(m,M,N,r,t) \in \Delta^{*}$ and $P_{m,r}(t)=\{8,124,182,211,240,269,356,414,501,530,559,588,646,762\}$.
By Lemma \ref{lem2}, we know that $\left\{\begin{bmatrix}
	1  &  0 \\
	\delta  &  1      
	\end{bmatrix}:\delta|87 \right\}$ forms a complete set of double coset representatives of $\Gamma_{0}(N) \backslash \Gamma/ \Gamma_\infty$.
	Let $\gamma_{\delta}=\begin{bmatrix}
	1  &  0 \\
	\delta  &  1      
	\end{bmatrix}$. Let $r'=(0,0,0,0)\in R(87)$ and we use $Sage$ to verify that
	$p_{m,r}(\gamma_{\delta})+p_{r'}^{*}(\gamma_{\delta}) \geq 0$ for each $\delta | N$. We compute that the upper bound in Lemma \ref{lem1} is $\lfloor\nu\rfloor=14$. Using $Sage$ we verify that $b_{21}(4(841n+t')+1)\equiv0\pmod{2}$ for all $t' \in P_{m,r}(t)$ and for $n\leq 14$. By Lemma \ref{lem1} we conclude that $b_{21}(4(841n+t')+1)\equiv0\pmod{2}$ for all $t' \in P_{m,r}(t)$ and for all $n\geq 0$. To prove the remaining congruences, we take $(m,M,N,r,t)=(841,3,87,(-1,4),443)$. It is easy to verify that $(m,M,N,r,t) \in \Delta^{*}$ and $P_{m,r}(t)=\{37,66,95,153,298,327,443,472,617,675,704,733,791,820\}$. Following similar steps as shown before, we find that $b_{21}(4(841n+t')+1)\equiv0\pmod{2}$ for all $t' \in P_{m,r}(t)$ and for all $n\geq 0$. This completes the proof of the theorem.
\end{proof}
\section{Proof of Theorems \ref{thm3} and \ref{thm4}} 
In order to prove Theorems \ref{thm3} and  \ref{thm4}, we first prove the following lemma.
\begin{lemma}
We have
\begin{align}\label{lem1.1}
\sum_{n=0}^{\infty}b_9(2n+1)q^n=\frac{f_2^2f_3f_{18}}{f_1^3f_6};
\end{align}
\begin{align}\label{lem1.2}
\sum_{n=0}^{\infty}b_9(4n)q^n\equiv\frac{f_3^7}{f_1f_9^2}\pmod2.
\end{align}
\end{lemma}
\begin{proof}
Letting $\ell=9$ in \eqref{gen_fun} we have
\begin{align}\label{lem1.3}
\sum_{n=0}^{\infty}b_9(n)q^n=\frac{f_9}{f_1}.
\end{align}
From Lemma 3.5 in \cite{Xia2012} we have
\begin{align}\label{lem1.4}
\frac{f_9}{f_1}=\frac{f_{12}^3f_{18}}{f_2^2f_6f_{36}}+q\frac{f_4^2f_6f_{36}}{f_2^3f_{12}}.
\end{align}
Extracting the terms with odd powers of $q$ and then using \eqref{lem1.4}, we obtain
\begin{align*}
\sum_{n=0}^{\infty}b_9(2n+1)q^n=\frac{f_2^2f_3f_{18}}{f_1^3f_6}.
\end{align*}
From \eqref{lem1.4}, extracting the terms with even powers of $q$ and then using \eqref{lem1.3}, we obtain
\begin{align}\label{lem1.5}
\sum_{n=0}^{\infty}b_9(2n)q^n=\frac{f_6^3f_9}{f_1^2f_3f_{18}}\equiv \frac{f_3^5}{f_1^2f_9}\pmod 2.
\end{align}
From \cite[(2.5)]{Kathiravan2021} we have
\begin{align}\label{lem1.6}
\frac{f_1^3}{f_3}=\frac{f_4^3}{f_{12}}-3q\frac{f_2^2f_{12}^3}{f_4f_6^2}.
\end{align}
Magnifying equation \eqref{lem1.6} by $q\rightarrow q^3$ and combining with \eqref{lem1.5}, we obtain 
\begin{align*}
\sum_{n=0}^{\infty}b_9(2n)q^n\equiv \frac{f_3^2}{f_1^2}\left(\frac{f_{12}^3}{f_{36}}+q^3\frac{f_6^2f_{36}^3}{f_{12}f_{18}^2}\right)\pmod 2.
\end{align*}
Extracting the terms with even powers of $q$, we obtain
\begin{align*}
\sum_{n=0}^{\infty}b_9(4n)q^n\equiv\frac{f_3^7}{f_1f_9^2}\pmod2.
\end{align*}
This completes the proof of the lemma. 
\end{proof}
\begin{proof}[Proof of Theorem \ref{thm3}] 
Let 
\begin{align*}
A(z) := \prod_{n=1}^{\infty} \frac{(1-q^{54n})^2}{(1-q^{108n})} = \frac{\eta^2(54z)}{\eta(108z)}. 
\end{align*}
Then using the binomial theorem we have 
\begin{align*}
A^{2^{k}}(z) = \frac{\eta^{2^{k+1}}(54z)}{\eta^{2^{k}}(108z)} \equiv 1 \pmod {2^{k+1}}.
\end{align*}
Define $B_{k}(z)$ by
\begin{align*}
B_{k}(z):= \left(\frac{\eta^{2}(6z)\eta(9z)\eta(54z)}{\eta^3(3z)\eta(18z)}\right)A^{2^{k}}(z)
=\frac{\eta^{2}(6z)\eta(9z)\eta^{2^{k+1}+1}(54z)}{\eta^3(3z)\eta(18z)\eta^{2^{k}}(108z)}.
\end{align*}
Modulo $2^{k+1}$, we have
\begin{align}\label{thm3.1}
B_{k}(z)\equiv\frac{\eta^{2}(6z)\eta(9z)\eta(54z)}{\eta^3(3z)\eta(18z)} =q^{2}\left(\frac{(q^{6}; q^{6})_{\infty}^2(q^{9}; q^{9})_{\infty}(q^{54}; q^{54})_{\infty}}{(q^{3}; q^{3})_{\infty}^3(q^{18}; q^{18})_{\infty}}\right).
\end{align}
Combining \eqref{lem1.1} and \eqref{thm3.1}, we obtain 
\begin{align}\label{thm3.2}
B_{k}(z) \equiv \sum_{n=0}^{\infty}b_9(2n+1)q^{3n+2} \pmod {2^{k+1}}.
\end{align}
Now, $B_{k}(z)$ is an eta-quotient with $N =324$. We next prove that $B_{k}(z)$ is a modular form for all $k\geq 6$.  We know that the cusps of $\Gamma_{0}(324)$ are represented by fractions $\frac{c}{d}$, where $d\mid 324$ and $\gcd(c, d)=1$. By Theorem \ref{thm_ono2}, we find that $B_{k}(z)$ is holomorphic at a cusp $\frac{c}{d}$ if and only if
\begin{align*}
&\left(2^{k+1}+1\right)\frac{\gcd(d,54)^2}{54}+2\frac{\gcd(d,6)^2}{6}+\frac{\gcd(d,9)^2}{9}-3\frac{\gcd(d,3)^2}{3}-\frac{\gcd(d,18)^2}{18}\\
&-2^{k}\frac{\gcd(d,108)^2}{108}\geq 0.
\end{align*}
Equivalently, if and only if	
\begin{align*}
L:=&~(2^{k+2}+2)G_1+ 36G_2 +12G_3-108 G_4-6G_5-2^{k}\geq 0,
\end{align*}
where $G_1=\displaystyle \frac{\gcd(d,54)^2}{\gcd(d,108)^2}, G_2=\displaystyle \frac{\gcd(d,6)^2}{\gcd(d,108)^2}, 
G_3=\displaystyle\frac{\gcd(d,9)^2}{\gcd(d,108)^2}$,  $
G_4=\displaystyle \frac{\gcd(d,3)^2}{\gcd(d,108)^2}$, and $G_5=\displaystyle\frac{\gcd(d,18)^2}{\gcd(d,108)^2}$ respectively.
\par 
We now consider the following four cases according to the divisors of $324$ and find the values of $G_i$ for $i=1,2,\ldots, 5$. Let $d$ be a divisor of $N=324$. \\
Case (i). For $d|324$ and $d\notin\{4,12,36,108,324\}$, we find that $G_1=1$, $1/81\leq G_2\leq 1, 1/{36}\leq G_3\leq 1$, $1/{324}\leq G_4\leq 1$ and $1/9\leq G_5\leq1$. Hence, 
\begin{align*}
L\geq 2^{k+2}+2+36/81+12/36-108-6-2^k=3\cdot2^k-112-7/9. 
\end{align*}
Since $k\geq 6$, we have $L\geq 0$.\\
Case (ii). For $d=4,12$, we find that $G_1=G_2=G_5=1/4$ and $G_3=G_4=1/16$. Hence, $L=2$.\\
Case (iii). For $d=36$, we find that $G_1=G_5=1/4$, $G_2=1/36$, $G_3=1/16$ and $ G_4=1/144$. Hence, the value of $L$ is equal to $0$.\\
Case (iv). For $d=108,324$, we find that $G_1=1/4$, $G_2=1/324$, $G_3=1/144$, $G_4=1/1296$ and $ G_5=1/36$. Hence, we have value of $L$ equal to $4/9$.\\
Hence, $B_{k}(z)$ is holomorphic at every cusp $\frac{c}{d}$ for all $k\geq 6$.
Using Theorem \ref{thm_ono1}, we find that the weight of $B_{k}(z)$ is equal to $2^{k-1}$. Also, the associated character for $B_{k}(z)$ is given by $\chi_1=(\frac{4\cdot 3^{3\cdot 2^k+2}}{\bullet})$.
This proves that $B_{k}(z) \in M_{2^{k-1}}\left(\Gamma_{0}(324), \chi_1\right)$ for all $k\geq 6$. Also, the Fourier coefficients of $B_{k}(z)$ are all integers. Hence by Theorem \ref{Serre}, the Fourier coefficients of $B_{k}(z)$ are almost always divisible by $m=2^k$, for any positive integer $k$. Due to \eqref{thm3.2}, the same holds for $b_9(2n+1)$. This completes the proof of the theorem. 
\end{proof}
We now prove Theorem \ref{thm4}. We recall the following classical result due to Landau \cite{Landau1908}.
\begin{lemma}\label{Landau}
	Let $r(n)$ and $s(n)$ be quadratic polynomials. Then
	\begin{align*}
	\left(\sum_{n\in\mathbb{Z}}q^{r(n)}\right)\left(\sum_{n\in\mathbb{Z}}q^{s(n)}\right)
	\end{align*}
	is lacunary modulo 2.
\end{lemma}
\begin{proof}[Proof of Theorem \ref{thm4}] 
	We first recall the following identity \cite[(7)]{Keith2021}:
	\begin{align*}
	f_1^3\equiv f_3+qf_9^3\pmod 2.
	\end{align*}
We rewrite the above identity as 
	\begin{align}\label{newthm4.1}
	\frac{f_3}{f_1}\equiv f_1^2+q\frac{f_9^3}{f_1}\pmod2.
	\end{align}
	Combining \eqref{lem1.2} and \eqref{newthm4.1}, we obtain 
	\begin{align}\label{newthm4.02}
	\sum_{n=0}^{\infty}b_9(4n)q^n\equiv \frac{f_3^6 f_1^2}{f_9^2}+q \frac{f_3^6 f_9^2}{f_1}\pmod2.
	\end{align}
	We note that the second term of \eqref{newthm4.02} is lacunary modulo $2$, by Theorem \ref{Cotron}. For the first term of \eqref{newthm4.02}, we again recall the following identity \cite[p. 12]{Keith2021} 
	\begin{align*}
	\frac{f_1^3}{f_3}\equiv 1+\sum_{n\in\mathbb{Z}}q^{(3n-1})^2\pmod 2,
	\end{align*}
	which is quadratic. Hence, the same holds true for $(f_3^3/f_9)^2$, by substituting $q$ with $q^6$. More precisely, we obtain
	\begin{align}\label{newthm4.3}
	\left(\frac{f_3^3}{f_9}\right)^2\equiv 1+\sum_{n\in\mathbb{Z}}q^{6(3n-1)^2}\pmod 2.
	\end{align}
	Now, squaring the Euler's Pentagonal Number formula, we have
	\begin{align}\label{newthm4.4}
	f_1^2\equiv \sum_{n\in\mathbb{Z}} q^{n(3n-1)}\pmod 2.
	\end{align}
	Finally combining \eqref{newthm4.3} and \eqref{newthm4.4}, and then applying Lemma \ref{Landau} we conclude that the first term of \eqref{newthm4.02} is also lacunary modulo 2. This completes the proof of the theorem.	
	\end{proof}
	\begin{remark}
Theorem \ref{thm4} can also be proved using the Serre's density result as shown in the proof of Theorem \ref{thm3}. For this, we rewrite \eqref{lem1.2} in terms of $\eta$-quotients and obtain 
\begin{align}\label{thm4.1}
\sum_{n=0}^{\infty}b_9(4n)q^{12n+1}\equiv\frac{\eta^7(36z)}{\eta(12z)\eta^2(108z)}\pmod2.
\end{align}
Let $F(z)=\frac{\eta^7(36z)}{\eta(12z)\eta^2(108z)}$. As shown in the proof of Theorem \ref{thm3}, one can prove that $F(z)\in M_{2}(\Gamma_{0}(1296), (\frac{2^8 3^{7}}{\bullet}))$. By Theorem \ref{Serre}, the Fourier coefficients of $F(z)$ are almost always divisible by $m=2$. Due to \eqref{thm4.1}, the same holds for $b_9(4n)$.
\end{remark}
\section{Acknowledgements}
We are extremely grateful to Professor Fabrizio Zanello for many helpful comments.

\end{document}